\documentclass[12pt]{amsart}
\usepackage{amssymb}
\usepackage[all]{xy}

\newtheorem{theorem}{Theorem}[section]
\newtheorem{definition}[theorem]{Definition}
\newtheorem{proposition}[theorem]{Proposition}

\begin{document}

\title[Homogeneous quantum groups]{Homogeneous quantum groups and their easiness level}

\author[Teo Banica]{Teo Banica}
\address{T.B.: Department of Mathematics, University of Cergy-Pontoise, F-95000 Cergy-Pontoise, France. {\tt teo.banica@gmail.com}}


\begin{abstract}
Given a closed subgroup $G\subset U_N^+$ which is homogeneous, in the sense that we have $S_N\subset G\subset U_N^+$, the corresponding Tannakian category $C$ must satisfy $span(\mathcal{NC}_2)\subset C\subset span(P)$. Based on this observation, we construct a certain integer $p\in\mathbb N\cup\{\infty\}$, that we call ``easiness level'' of $G$. The value $p=1$ corresponds to the case where $G$ is easy, and we explore here, with some theory and examples, the case $p>1$. As a main application, we show that $S_N\subset S_N^+$ and other liberation inclusions, known to be maximal in the easy setting, remain maximal at the easiness level $p=2$ as well.
\end{abstract}

\maketitle

\section*{Introduction}

The easy quantum groups were introduced in our joint paper with Speicher \cite{bsp}, following some previous work with Bichon and Collins \cite{bbc}. The idea is very simple. Given a closed subgroup $G\subset U_N^+$ which is homogeneous, in the sense that we have $S_N\subset G\subset U_N^+$, the corresponding Tannakian category $C$ must appear as follows:
$$span(\mathcal{NC}_2)\subset C\subset span(P)$$

Here $P,{\mathcal NC}_2$ are respectively the categories of partitions, and of the matching noncrossing pairings, which are known to correspond, via Brauer type results, to $S_N,U_N^+$. As for the result itself, this follows from Woronowicz's Tannakian theory in \cite{wo2}.

Based on this fact, the idea in \cite{bsp} was to call $G$ ``easy'' when $C$ appears in the simplest possible way: $C=span(D)$, for a certain category of partitions ${\mathcal NC}_2\subset D\subset P$. Such quantum groups can be investigated with various combinatorial tools, partly coming from Voiculescu's free probability theory \cite{vdn}, and a whole theory, featuring several non-trivial structure and classification results, was built in this way. See \cite{bcs}, \cite{rw3}, \cite{tw1}.

Going beyond easiness is a tricky task, and several attempts have been made, over the last years \cite{cwe}, \cite{fr1}, \cite{fr2}, \cite{swe}. Our aim here is to present one more such attempt.

To be more precise, to any homogeneous quantum group $S_N\subset G\subset U_N^+$ we will associate a certain integer $p\in\mathbb N\cup\{\infty\}$, that we call ``easiness level'' of $G$. The value $p=1$ corresponds to the case where $G$ is easy, and we will discuss here the case $p>1$. As a main application, we will show that certain liberation inclusions $G\subset G^\times$, which are maximal in the easy setting, remain maximal at the easiness level 2 as well.

The construction of the easiness level will be done as follows:
\begin{enumerate}
\item Our first observation is that any homogeneous quantum group $S_N\subset G\subset U_N^+$ has an ``easy envelope'', that we denote by $G^1$. This easy envelope is the smallest intermediate easy quantum group $G\subset G^1\subset U_N^+$, and its category of partitions consists of the partitions $\pi\in P$ which belong to the Tannakian category of $G$.

\item More generally, we can consider the linear combinations of type $\alpha_1\pi_1+\ldots+\alpha_p\pi_p$, of fixed length $p\in\mathbb N$, which belong to the Tannakian category of $G$. These combinations do not form a Tannakian category, but we can consider the Tannakian category generated by them, and we obtain in this way a quantum group $G^p$.

\item While the construction $G\to G^p$ is something quite abstract at $p\geq2$, we can still study it, by using various abstract Tannakian methods. Our main theoretical result here is that the quantum groups $G^p$ constructed above from a decreasing family, $G_1\supset G_2\supset G_3\supset\ldots\supset G$, whose intersection is $G$.

\item Based on these facts, we will define the easiness level of $G$ to be the smallest number $p\in\mathbb N\cup\{\infty\}$ such that we have $G=G^p$, with the convention that the value $p=\infty$ corresponds to the case where $G\neq G^p$, for any $p\in\mathbb N$. As an illustration, the case $p=1$ corresponds to the case where $G$ is easy.
\end{enumerate}

In addition to the above facts, we will prove that we have $p\leq B_r$, where $r$ is the presentation level of the discrete quantum group dual $\Gamma=\widehat{G}$, and $B_r$ is the $r$-th Bell number. Once again, this is something that follows from Tannakian duality.

At the level of examples now, the situation is quite interesting, and there is definitely work to be done. Many of the known examples of homogeneous quantum groups $S_N\subset G\subset U_N^+$ are in fact easy, but we have some non-easy examples as well, as follows:
\begin{enumerate}
\item Let $U_N^d=\{g\in U_N|\det g\in\mathbb Z_d\}$, where $\mathbb Z_d$ is the group of $d$-th roots of unity. When $d$ is even we have a homogeneous group, $S_N\subset U_N^d\subset U_N$, and we will show that the enveloping easy group is $U_N$. The combinatorics is quite interesting, related to Woronowicz' computations in \cite{wo2}, for $SU_N=U_N^1$, and its deformations.

\item Let $H_N^s=\mathbb Z_s\wr S_N$, and $H_N^{s,d}=\{g\in H_N^s|\det g\in\mathbb Z_d\}$. The group $H_N^s$ is known from \cite{fbl} to be easy, and we will investigate here the easiness properties of its subgroup $H_N^{s,d}=H_N^s\cap U_N^d$. Once again, the combinatorics is interesting, related to \cite{fbl}. In addition, $H_N^{s,d}$ plays an important role in reflection group theory \cite{sto}.

\item At the quantum group level now, we have a construction, which is new. The idea is that, according to \cite{bsp}, \cite{rau}, we have an inclusion $B_N\subset B_N^+$, an isomorphism $B_N^+\simeq O_{N-1}^+$, and inclusions $O_{N-1}\subset O_{N-1}^*\subset O_{N-1}^+$.  Thus by taking the image of $O_{N-1}^*$ inside $B_N^+$ we obtain a quantum group $B_N^\circ$, that we will study here.

\item There is as well a complex analogue of the above construction to be studied, with $B_N$ replaced by the complex bistochastic group $C_N$, from \cite{tw1}, \cite{tw2}. By using the same method we obtain a quantum group $C_N^\circ$, that we will study. In fact, since $U_N^*$ is not ``unique'', we will obtain in this way several new quantum groups.
\end{enumerate}

Summarizing, talking about the easiness level leads us into looking at the non-easy examples of homogeneous quantum groups $S_N\subset G\subset U_N^+$, and their combinatorics. The whole subject is definitely interesting, and we will do some exploration work here.

At the level of the applications now, our idea will be that of investigating maximality questions, for inclusions of easy quantum groups. With suitable definitions, it is known from \cite{bcs}, \cite{bsp}, \cite{rw3} that there are precisely $4$ ``true'' liberations of orthogonal easy quantum groups, with the intermediate liberations, in the easy framework, as follows:
\begin{enumerate}
\item $S_N\subset S_N^+$, with no intermediate object.

\item $H_N\subset H_N^+$, with uncountably many intermediate objects.

\item $O_N\subset O_N^+$, with $O_N^*$ as unique intermediate object.

\item $B_N\subset B_N^+$, with no intermediate object.
\end{enumerate}

These inclusions are all very interesting. A well-known conjecture, going back to our paper with Bichon \cite{bb1}, and which is perhaps the most important open question regarding the quantum permutation groups, states that $S_N\subset S_N^+$ is maximal. Regarding now $H_N\subset H_N^+$, this inclusion plays a key role in the classification of the easy quantum groups, as shown by Raum and Weber in their work \cite{rw1}, \cite{rw2}, \cite{rw3}. The inclusions $O_N\subset O_N^*\subset O_N^+$ are quite subtle too, and as explained in \cite{max}, their study can effectively be done, and corresponds somehow to a ``warm-up'' for the study of the $S_N\subset S_N^+$ conjecture. As for the inclusion $B_N\subset B_N^+$, this is quite interesting too, in view of the above-mentioned intermediate quantum group $B_N\subset B_N^\circ\subset B_N^+$ that we construct here.

Thus, we have here yet another rich landscape of open problems, this time rather well-known to specialists, but still waiting to be further investigated. We will present here a few contributions to this subject, as follows:
\begin{enumerate}
\item According to \cite{bb1} we have $S_4^+=SO_3^{-1}$, the subgroups $G\subset S_4^+$ are subject to an ADE type classification, and in particular, $S_4\subset S_4^+$ follows to be maximal. By using some recent advances from subfactor theory, from \cite{imp}, \cite{jms}, we will prove here that the inclusion $S_5\subset S_5^+$ is maximal as well.

\item It is known from \cite{max} that the inclusion $O_N\subset O_N^*$ is maximal, and in connection to this, we have two remarks. First, the results in \cite{bdu} allow in principle to simplify the proof in \cite{max}, and this remains to be done. And second, it follows from \cite{max} that the inclusion $B_N\subset B_N^\circ$ that we construct here is maximal as well.

\item We will investigate the inclusion $S_N\subset S_N^+$, by using our notion of easiness level. To be more precise, the result from \cite{bcs}, stating that $S_N\subset S_N^+$ is maximal in the easy setting, tells us that this inclusion is maximal at order $p=1$. We will show here that this inclusion is maximal at level $p=2$ as well.

\item Finally, we will present similar results, in terms of the easiness level, for the inclusions $O_N\subset O_N^*\subset O_N^+$, and $B_N\subset B_N^+$. Regarding the inclusion $H_N\subset H_N^+$, which has uncountably many intermediate objects, the situation here is of course considerably more complicated, and we have no advances on it.
\end{enumerate}

The paper is organized as follows: in 1-2 we discuss the homogeneous quantum groups and their easiness level, with some general results, in 3-4 and 5-6 we study some non-trivial examples, coming from the classical unitary and reflection groups, and from certain half-liberations, and in 7-8 we discuss the various maximality questions for the liberation inclusions of easy groups, notably by using our notion of easiness level.

\bigskip

\noindent {\bf Acknowledgements.} I would like to thank J. Bichon, B. Collins, S. Curran for many discussions on maximality questions, a few years ago. Thanks to Poulette, too.

\section{Homogeneous quantum groups}

We use Woronowicz's quantum group formalism in \cite{wo1}, \cite{wo2}, under the extra assumption $S^2=id$. To be more precise, the definition that we will need is:

\begin{definition}
Assume that $(A,u)$ is a pair consisting of a $C^*$-algebra $A$, and a unitary matrix $u\in M_N(A)$ whose coefficients generate $A$, such that the formulae
$$\Delta(u_{ij})=\sum_ku_{ik}\otimes u_{kj}\quad,\quad \varepsilon(u_{ij})=\delta_{ij}\quad,\quad S(u_{ij})=u_{ji}^*$$
define morphisms of $C^*$-algebras $\Delta:A\to A\otimes A$, $\varepsilon:A\to\mathbb C$, $S:A\to A^{opp}$. We write then $A=C(G)$, and call $G$ a compact matrix quantum group.
\end{definition}

The basic examples are the compact Lie groups, $G\subset U_N$. Indeed, given such a group we can set $A=C(G)$, and let $u_{ij}:G\to\mathbb C$ be the standard coordinates, $u_{ij}(g)=g_{ij}$. The axioms are then satisfied, with $\Delta,\varepsilon,S$ being the functional analytic transposes of the multiplication $m:G\times G\to G$, unit map $u:\{.\}\to G$, and inverse map $i:G\to G$.

There are many other examples. For instance given a finitely generated discrete group $\Gamma=<g_1,\ldots,g_N>$ we can set $A=C^*(\Gamma)$, and $u=diag(g_1,\ldots,g_N)$. The axioms are once again satisfied, and the resulting quantum group is denoted $G=\widehat{\Gamma}$. See \cite{wo1}.

The following key construction is due to Wang \cite{wa1}:

\begin{proposition}
We have a compact quantum group $U_N^+$, defined via
$$C(U_N^+)=C^*\left((u_{ij})_{i,j=1,\ldots,N}\Big|u^*=u^{-1},u^t=\bar{u}^{-1}\right)$$
and the compact matrix quantum groups are precisely the closed subgroups $G\subset U_N^+$.
\end{proposition}

\begin{proof}
It is routine to check that if $u=(u_{ij})$ is biunitary ($u^*=u^{-1},u^t=\bar{u}^{-1}$), then so are the matrices $u^\Delta=(\sum_ku_{ik}\otimes u_{kj})$, $u^\varepsilon=(\delta_{ij})$, $u^S=(u_{ji}^*)$. Thus we can construct $\Delta,\varepsilon,S$ as in Definition 1.1, by using the universal property of $C(U_N^+)$.

Regarding the last assertion, in the context of Definition 1.1 we have $u^*=u^{-1}$, and by applying $S$ we obtain $u^t=\bar{u}^{-1}$. Thus $u$ is biunitary, so we have a quotient map $C(U_N^+)\to C(G)$, which corresponds to an inclusion of quantum groups $G\subset U_N^+$. 
\end{proof}

Consider the standard action $S_N\curvearrowright\mathbb C^N$, obtained by permuting the coordinates, $\sigma(e_i)=e_{\sigma(i)}$. This action provides us with an embedding $S_N\subset U_N$, and so with an embedding $S_N\subset U_N^+$. We are interested here in the following quantum groups:

\begin{definition}
A closed subgroup $G\subset U_N^+$ is called homogeneous if it appears as
$$S_N\subset G\subset U_N^+$$
where the inclusion $S_N\subset U_N^+$ comes via the standard permutation matrices.
\end{definition}

The idea will be that of investigating such quantum groups by using Tannakian duality techniques. Let us first recall that asssociated to a closed subgroup $G\subset U_N^+$ is its Tannakian category $C=(C(k,l))$, formed by the following linear spaces:
$$C_{kl}=Hom(u^{\otimes k},u^{\otimes l})$$

Here the exponents $k,l$ are by definition colored integers, $\circ\bullet\bullet\circ\ldots$, with the corresponding powers of $u$ being given by $u^\circ=u,u^\bullet=\bar{u}$ and multiplicativity. 

Now let $P(k,l)$ be the set of partitions between an upper row of points representing $k$, and a lower row of points representing $l$. We will represent these partitions as pictures, with $|^{\hskip-1.3mm\circ}_{\hskip-1.3mm\circ},|^{\hskip-1.3mm\bullet}_{\hskip-1.3mm\bullet},{\ }_\circ\hskip-1.25mm\cap_{\!\!\bullet},{\ }_\bullet\hskip-1.25mm\cap_{\!\!\circ}$ belonging for instance to $P(\circ,\circ)$, $P(\bullet,\bullet)$, $P(\emptyset,\circ\bullet)$, $P(\emptyset,\bullet\circ)$.

Following \cite{bsp}, \cite{tw1}, let us introduce:

\begin{definition}
A category of partitions is a family of subsets $D(k,l)\subset P(k,l)$ containing the identity and duality partitions $|^{\hskip-1.3mm\circ}_{\hskip-1.3mm\circ},|^{\hskip-1.3mm\bullet}_{\hskip-1.3mm\bullet},{\ }_\circ\hskip-1.25mm\cap_{\!\!\bullet},{\ }_\bullet\hskip-1.25mm\cap_{\!\!\circ}$, and which is stable under:
\begin{enumerate}
\item The horizontal concatenation operation, $(\pi,\sigma)\to[\pi\sigma]$.

\item The vertical concatenation, with the middle components erased, $(\pi,\sigma)\to[^\sigma_\pi]$. 

\item The upside-down turning, with switching of the colors, $\pi\to\pi^*$. 
\end{enumerate}
\end{definition}

As a basic example, we have $P$ itself. Another basic example is the category ${\mathcal NC}_2$ of noncrossing pairings which are ``matching'', in the sense that the horizontal strings must connect $\circ-\bullet$, and the vertical strings must connect $\circ-\circ$ or $\bullet-\bullet$. Observe that we have ${\mathcal NC}_2=<|^{\hskip-1.3mm\circ}_{\hskip-1.3mm\circ},|^{\hskip-1.3mm\bullet}_{\hskip-1.3mm\bullet},{\ }_\circ\hskip-1.25mm\cap_{\!\!\bullet},{\ }_\bullet\hskip-1.25mm\cap_{\!\!\circ}>$. Thus, if $D$ is a category of partitions, then ${\mathcal NC}_2\subset D\subset P$.

The relation with the quantum groups comes from the following construction:

\begin{definition}
Associated to a partition $\pi\in P(k,l)$ is the linear map
$$T_\pi(e_{i_1}\otimes\ldots\otimes e_{i_k})=\sum_{j_1\ldots j_l}\delta_\pi\begin{pmatrix}i_1&\ldots&i_k\\ j_1&\ldots&j_l\end{pmatrix}e_{j_1}\otimes\ldots\otimes e_{j_l}$$
where $e_1,\ldots,e_N$ is the standard basis of $\mathbb C^N$, and $\delta_\pi\in\{0,1\}$ is a Kronecker symbol.
\end{definition}

To be more precise, here the Kronecker symbol $\delta_\pi$ takes by definition the value 1 when each block of $\pi$ contains equal indices, and takes the value 0, otherwise.

As explained in \cite{bsp}, the correspondence $\pi\to T_\pi$ has a number of remarkable categorical properties, summarized in the following formulae:
$$T_{[\pi\sigma]}=T_\pi\otimes T_\sigma\quad,\quad T_{[^\sigma_\pi]}\sim T_\pi T_\sigma\quad,\quad T_{\pi^*}=T_\pi^*$$

With these ingredients in hand, we can now formulate:

\begin{definition}
A homogeneous quantum group $S_N\subset G\subset U_N^+$ is called easy when 
$$Hom(u^{\otimes k},u^{\otimes l})=span\left(T_\pi\Big|\pi\in D(k,l)\right)$$
for any colored integers $k,l$, for a certain category of partitions ${\mathcal NC}_2\subset D\subset P$.
\end{definition}

As basic examples, both $S_N,U_N^+$ are known to be easy, the corresponding categories of partitions being respectively $P,{\mathcal NC}_2$. There are many other examples, coming from the fact that each category of partitions ${\mathcal NC}_2\subset D\subset P$ produces, via Woronowicz' Tannakian duality results in \cite{wo2}, a certain homogeneous quantum group $S_N\subset G\subset U_N^+$. See \cite{bsp}.

We will be mainly interested here in the non-easy case, and we will need:

\begin{theorem}
The homogeneous quantum groups $S_N\subset G\subset U_N^+$ are in one-to-one correspondence with the intermediate tensor categories
$$span\left(T_\pi\Big|\pi\in\mathcal{NC}_2\right)\subset C\subset span\left(T_\pi\Big|\pi\in P\right)$$
where $P$ is the category of all partitions, ${\mathcal NC}_2$ is the category of the matching noncrossing pairings, and $\pi\to T_\pi$ is the construction in Definition 1.5.
\end{theorem}

\begin{proof}
This follows from Woronowicz' Tannakian duality results in \cite{wo2}, and from the above-mentioned Brauer type results for $S_N,U_N^+$. To be more precise, we know from \cite{wo2}, or rather from the ``soft'' form of the duality, from \cite{mal}, that each closed subgroup $G\subset U_N^+$ can be reconstructed from its Tannakian category $C=(C(k,l))$, as follows:
$$C(G)=C(U_N^+)\Big/\left<T\in Hom(u^{\otimes k},u^{\otimes l})\Big|\forall k,l,\forall T\in C(k,l)\right>$$

Thus we have a one-to-one correspondence $G\leftrightarrow C$, and since the endpoints $G=S_N,U_N^+$ are both easy, corresponding to the categories $C=span(T_\pi|\pi\in D)$ with $D=P,{\mathcal NC}_2$, this gives the result. For full details regarding all this, see \cite{bsp}, \cite{tw1}.
\end{proof}

\section{The easiness level}

Our purpose in what follows will be that of using the Tannakian result in Theorem 1.7 above, in order to introduce and study a combinatorial notion of ``easiness level'', for the arbitrary intermediate quantum groups $S_N\subset G\subset U_N^+$. 

Let us begin with the following simple fact:

\begin{proposition}
Given a homogeneous quantum group $S_N\subset G\subset U_N^+$, with associated Tannakian category $C=(C(k,l))$, the sets
$$D^1(k,l)=\left\{\pi\in P(k,l)\Big|T_\pi\in C(k,l)\right\}$$ 
form a category of partitions, in the sense of Definition 1.4.
\end{proposition}

\begin{proof}
We use the basic categorical properties of the correspondence $\pi\to T_\pi$, namely:
$$T_{[\pi\sigma]}=T_\pi\otimes T_\sigma\quad,\quad T_{[^\sigma_\pi]}\sim T_\pi T_\sigma\quad,\quad T_{\pi^*}=T_\pi^*$$

Together with the fact that $C$ is a tensor category, we deduce from these formulae that we have the following implications:
\begin{eqnarray*}
\pi,\sigma\in D^1&\implies&T_\pi,T_\sigma\in C\implies T_\pi\otimes T_\sigma\in C\implies T_{[\pi\sigma]}\in C\implies[\pi\sigma]\in D^1\\
\pi,\sigma\in D^1&\implies&T_\pi,T_\sigma\in C\implies T_\pi T_\sigma\in C\implies T_{[^\sigma_\pi]}\in C\implies[^\sigma_\pi]\in D^1\\
\pi\in D^1&\implies&T_\pi\in C\implies T_\pi^*\in C\implies T_{\pi^*}\in C\implies\pi^*\in D^1
\end{eqnarray*}

Thus $D^1$ is indeed a category of partitions, in the sense of \cite{bsp}, as claimed.
\end{proof}

We can further refine the above observation, in the following way:

\begin{proposition}
Given a quantum group $S_N\subset G\subset U_N^+$, construct $D^1\subset P$ as above, and let $S_N\subset G^1\subset U_N^+$ be the easy quantum group associated to $D^1$. Then:
\begin{enumerate}

\item We have $G\subset G^1$, as subgroups of $U_N^+$.

\item $G^1$ is the smallest easy quantum group containing $G$.

\item $G$ is easy precisely when $G\subset G^1$ is an isomorphism.
\end{enumerate}
\end{proposition}

\begin{proof}
All this is elementary, the proofs being as follows:

(1) We know that the Tannakian category of $G^1$ is given by:
$$C_{kl}^1=span\left(T_\pi\Big|\pi\in D^1(k,l)\right)$$

Thus we have $C^1\subset C$, and so $G\subset G^1$, as subgroups of $U_N^+$.

(2) Assuming that we have $G\subset G'$, with $G'$ easy, coming from a Tannakian category $C'=span(D')$, we must have $C'\subset C$, and so $D'\subset D^1$. Thus, $G^1\subset G'$, as desired.

(3) This is a trivial consequence of (2).
\end{proof}

Summarizing, we have now a notion of ``easy envelope'', as follows:

\begin{definition}
The easy envelope of a homogeneous quantum group $S_N\subset G\subset U_N^+$ is the easy quantum group $S_N\subset G^1\subset U_N^+$ associated to the category of partitions
$$D^1(k,l)=\left\{\pi\in P(k,l)\Big|T_\pi\in C(k,l)\right\}$$ 
where $C=(C(k,l))$ is the Tannakian category of $G$.
\end{definition}

At the level of the examples, most of the known homogeneous quantum groups $S_N\subset G\subset U_N^+$ are in fact easy. However, there are many non-easy examples as well, and we will compute the easy envelopes in several cases of interest, in sections 3-6 below. 

As a technical observation now, we can in fact generalize the above construction to any closed subgroup $G\subset U_N^+$, and we have the following result:

\begin{proposition}
Given a closed subgroup $G\subset U_N^+$, construct $D^1\subset P$ as above, and let $S_N\subset G^1\subset U_N^+$ be the easy quantum group associated to $D^1$. We have then
$$G^1=(<G,S_N>)^1$$
where $<G,S_N>\subset U_N^+$ is the smallest closed subgroup containing $G,S_N$.
\end{proposition}

\begin{proof}
It is well-known, and elementary to show, using Woronowicz's Tannakian duality results in \cite{wo2}, that the smallest subgroup $<G,S_N>\subset U_N^+$ from the statement exists indeed, and can be obtained by intersecting the Tannakian categories of $G,S_N$:
$$C_{<G,S_N>}=C_G\cap C_{S_N}$$

We conclude from this that for any $\pi\in P(k,l)$ we have:
$$T_\pi\in C_{<G,S_N>}(k,l)\iff T_\pi\in C_G(k,l)$$

It follows that the $D^1$ categories for the quantum groups $<G,S_N>$ and $G$ coincide, and so the easy envelopes $(<G,S_N>)^1$ and $G^1$ coincide as well, as stated.
\end{proof}

In order now to fine-tune all this, by using an arbitrary parameter $p\in\mathbb N$, which can be thought of as being an ``easiness level'', we can proceed as follows:

\begin{definition}
Given a quantum group $S_N\subset G\subset U_N^+$, and an integer $p\in\mathbb N$, we construct the family of linear spaces
$$E^p(k,l)=\left\{\alpha_1T_{\pi_1}+\ldots+\alpha_pT_{\pi_p}\in C(k,l)\Big|\alpha_i\in\mathbb C,\pi_i\in P(k,l)\right\}$$
and we denote by $C^p$ the smallest tensor category containing $E^p=(E^p(k,l))$, and by $S_N\subset G^p\subset U_N^+$ the quantum group corresponding to this category $C^p$.
\end{definition}

As a first observation, at $p=1$ we have $C^1=E^1=span(D^1)$, where $D^1$ is the category of partitions constructed in Proposition 2.1. Thus the quantum group $G^1$ constructed above coincides with the ``easy envelope'' of $G$, from Definition 2.3 above.

In the general case, $p\in\mathbb N$, the family $E^p=(E^p(k,l))$ constructed above is not necessarily a tensor category, but we can of course consider the tensor category $C^p$ generated by it, as indicated. Finally, in the above definition we have used of course Woronowicz's Tannakian duality results in \cite{wo2}, in order to perform the operation $C^p\to G^p$.

In practice, the construction in Definition 2.5 is often something quite complicated, and it is convenient to use the following observation:

\begin{proposition}
The category $C^p$ constructed above is generated by the spaces
$$E^p(l)=\left\{\alpha_1T_{\pi_1}+\ldots+\alpha_pT_{\pi_p}\in C(l)\Big|\alpha_i\in\mathbb C,\pi_i\in P(l)\right\}$$
where $C(l)=C(0,l),P(l)=P(0,l)$, with $l$ ranging over the colored integers.
\end{proposition}

\begin{proof}
We use the well-known fact, from \cite{ntu}, \cite{wo1}, that given a closed subgroup $G\subset U_N^+$, we have a Frobenius type isomorphism $Hom(u^{\otimes k},u^{\otimes l})\simeq Fix(u^{\otimes\bar{k}l})$. If we apply this to the quantum group $G^p$ from Definition 2.5, we obtain an isomorphism $C(k,l)\simeq C(\bar{k}l)$.

On the other hand, we have as well an isomorphism $P(k,l)\simeq P(\bar{k}l)$, obtained by performing a counterclockwise rotation to the partitions $\pi\in P(k,l)$. According to the above definition of the spaces $E^p(k,l)$, this induces an isomorphism $E^p(k,l)\simeq E^p(\bar{k}l)$. 

We deduce from this that for any partitions $\pi_1,\ldots,\pi_p\in C(k,l)$, having rotated versions $\rho_1,\ldots,\rho_p\in C(\bar{k}l)$, and for any scalars $\alpha_1,\ldots,\alpha_p\in\mathbb C$, we have:
$$\alpha_1T_{\pi_1}+\ldots+\alpha_pT_{\pi_p}\in C(k,l)\iff\alpha_1T_{\rho_1}+\ldots+\alpha_pT_{\rho_p}\in C(\bar{k}l)$$

But this gives the conclusion in the statement, and we are done.
\end{proof}

The main properties of the construction $G\to G^p$ can be summarized as follows:

\begin{theorem}
Given a quantum group $S_N\subset G\subset U_N^+$, the quantum groups $G^p$ constructed above form a decreasing family, whose intersection is $G$:
$$G=\bigcap_{p\in\mathbb N}G^p$$
Moreover, $G$ is easy when this decreasing limit is stationary, $G=G^1$.
\end{theorem}

\begin{proof}
By definition of $E^p(k,l)$, and by using Proposition 2.2, these linear spaces form an increasing filtration of $C(k,l)$. The same remains true when completing into tensor categories, and so we have an increasing filtration, as follows:
$$C=\bigcup_{p\in\mathbb N}C^p$$

At the quantum group level now, we obtain the decreasing intersection in the statement. Finally, the last assertion is clear from Proposition 2.2.
\end{proof}

As a main consequence of the above results, we can now formulate: 

\begin{definition}
We say that a quantum group $S_N\subset G\subset U_N^+$ is easy at order $p$ when $G=G^p$. When $p$ is chosen minimal, we also say that $G$ has easiness level $p$.
\end{definition}

Observe that the order 1 notion corresponds to the usual easiness. In general, all this is quite abstract, and requires some explicit examples, in order to be understood. 

\section{Unitary groups}

In order to work out some explicit examples, let us first look at the classical case, $S_N\subset G\subset U_N$. The first thought goes to $SU_N$, but this group fails to be homogeneous, because it contains the alternating group $A_N$, but not $S_N$ itself. However, we have:

\begin{proposition}
Given a number $d\in\mathbb N\cup\{\infty\}$, consider the group
$$U_N^d=\left\{g\in U_N\Big|\det g\in\mathbb Z_d\right\}$$
where $\mathbb Z_d$ is the group of $d$-th roots of unity. If $2|d$, this group is homogeneous.
\end{proposition}

\begin{proof}
We recall from section 1 above that the embedding $S_N\subset U_N$ that we use is the one given by the usual permutation matrices, $\sigma(e_i)=e_{\sigma(i)}$. Thus the determinant of a permutation $\sigma\in S_N$ is its signature, $\varepsilon(\sigma)\in\mathbb Z_2$, and this gives the result.
\end{proof}

In what follows we will be mostly interested in the case $2|d$. However, the value $d=1$ is interesting and useful as well, because we have inclusions, as follows:
$$SU_N=U_N^1\subset U_N^d\subset U_N^\infty=U_N$$

By functoriality, we therefore obtain inclusions of categories, as follows:
$$C_{U_N}\subset C_{U_N^d}\subset C_{SU_N}$$

The group $U_N$ is well-known to be easy, its category being given by  $C_{U_N}=span(\mathcal P_2)$, where $\mathcal P_2$ is the category of the matching pairings. The representation theory of $SU_N$ is well-known as well, in diagrammatic terms, for instance from \cite{wo2}. 

Regarding now $U_N^d$, with $d\in\mathbb N\cup\{\infty\}$ being arbitrary, we have here:

\begin{proposition}
The Tannakian category of $U_N^d$ appears as a part of the Tannakian category of $SU_N$, obtained by restricting the attention to the spaces $C(k,l)$ with $\underline{k}=\underline{l}(d)$, where $\underline{k}$ is the number $\#\circ-\#\bullet$, computed over all the symbols of $k$.
\end{proposition}

\begin{proof}
Our first claim is that in the finite case, $d<\infty$, we have a disjoint union decomposition as follows, where $w=e^{2\pi i/Nd}$:
$$U_N^d=SU_N\ \sqcup\ wSU_N\ \sqcup\ w^2SU_N\ \sqcup\ldots\sqcup\ w^{d-1}SU_N$$

Indeed, we have $w^N=e^{2\pi i/d}$, and so the condition $\det g\in\mathbb Z_d$ from Proposition 3.1 means $\det g=w^{Nk}$, for some $k\in\{0,1,\ldots,d-1\}$, and our claim follows from:
$$\det g=w^{Nk}\iff \det\left(\frac{g}{w^k}\right)=1\iff\frac{g}{w^k}\in SU_N\iff g\in w^kSU_N$$

Now given $g\in U_N$, $\xi\in(\mathbb C^N)^{\otimes k}$ and $\lambda\in\mathbb C$, consider the following conditions:
$$g^{\otimes k}\xi=\xi\quad,\quad (\lambda g)^{\otimes k}\xi=\xi\quad,\ \ldots\ , 
\quad (\lambda^{d-1}g)^{\otimes k}\xi=\xi$$

These conditions are then equivalent to $g^{\otimes k}\xi=\xi$ and $\lambda^k=1$. Now by taking $g\in SU_N$ and $\lambda=w^N$, with $w=e^{2\pi i/Nd}$ being as above, this gives the result.

Finally, the assertion at $d=\infty$ can be proved in a similar way.
\end{proof}

Summarizing, the Tannakian category of $U_N^d$ appears as a part of the category computed in \cite{wo2}, and the value $d=\infty$, corresponding to $U_N$ itself, which is easy, is special. It is of course possible to go beyond this remark, but we will not need this here.

The easy envelope of $U_N^d$ can be computed as follows:

\begin{proposition}
The easy envelope of $U_N^d$ is $U_N$, for any $d\geq1$.
\end{proposition}

\begin{proof}
By functoriality, we can restrict the attention to the case of $U_N^1=SU_N$. We have to prove that the following implication holds:
$$\pi\in P(k),\xi_\pi\in Fix(g^{\otimes k}),\forall g\in SU_N\implies\pi\in\mathcal P_2(k)$$

For this purpose, we will basically use the isomorphism of projective versions $PSU_N=PU_N$. To be more precise, let us start with the following simple fact:
$$g^{\otimes k}\xi_\pi=\xi_\pi\implies(wg)^{\otimes k}\xi_\pi=w^k\xi_\pi,\forall w\in\mathbb T$$

In relation with the above implication, we have two cases, as follows:

\underline{Case $\underline{k}=0$}. Here the condition $\underline{k}=0$ means by definition that $k$ has the same number of black and white legs. Thus in the above formula we have $w^k=1$, and we obtain:
$$g^{\otimes k}\xi_\pi=\xi_\pi,\forall g\in SU_N\implies h^{\otimes k}\xi_\pi=\xi_\pi,\forall h\in U_N$$

We can therefore conclude by using the Brauer result for $U_N$, which states that the vectors $\xi_\pi$ on the right are those appearing from the partitions $\pi\in\mathcal P_2(k)$.

\underline{Case $\underline{k}\neq0$}. Here we must prove that a partition $\pi\in P(k)$ as above does not exist. In order to do so, observe first that, since $w^{\underline{k}}=\bar{w}^k$, we obtain:
$$g^{\otimes k}\xi_\pi=\xi_\pi,\forall g\in SU_N\implies h^{\otimes k\bar{k}}(\xi_\pi\otimes\xi_\pi)=(\xi_\pi\otimes\xi_\pi),\forall h\in U_N$$

But this shows that $\xi_\pi\otimes\xi_\pi$ must come from a pairing, and so $\xi_\pi$ itself must come from a pairing. Thus, as a first conclusion, we must have $\pi\in P_2(k)$.

Since the standard coordinates $u_{ij}$ of our group $SU_N$ commute, we can permute if we want the legs of this pairing, and we are left with a pairing of type $\pi=\cap\cap\ldots\cap$. Now if we take into account the labels, by further permuting the legs we can assume that we are in the case $\pi=[\alpha\beta\gamma]$, where $\alpha,\beta,\gamma$ are all pairings of type $\cap\cap\ldots\cap$, with $\alpha$ being white, $\beta$ being black, and $\gamma$ being matching. Moreover, by using the Brauer result for $U_N$, the invariance condition is trivially satisfied for $\gamma$, so we can assume $\gamma=\emptyset$.

Summarizing, we are now in the case $\pi=[\alpha\beta]$, with $\alpha,\beta$ being both of type $\cap\cap\ldots\cap$, and with $\alpha$ being white, and $\beta$ being black. With $\alpha=2r$ and $\beta=2s$, we have:
$$\xi_\pi=\sum_{i_1\ldots i_r}\sum_{j_1\ldots j_s}e_{i_1}\otimes e_{i_1}\otimes\ldots\otimes e_{i_r}\otimes e_{i_r}\otimes e_{j_1}\otimes e_{j_1}\otimes\ldots\otimes e_{j_s}\otimes e_{j_s}$$

An arbitrary matrix $g\in SU_N$ acts in the following way on this vector:
\begin{eqnarray*}
g^{\otimes k}\xi_\pi
&=&\sum_{i_1\ldots i_r}\sum_{j_1\ldots j_s}(gg^t)_{a_1b_1}\ldots(gg^t)_{a_rb_r}(\bar{g}g^*)_{c_1d_1}\ldots(\bar{g}g^*)_{c_sd_s}\\
&&e_{a_1}\otimes e_{b_1}\otimes\ldots\otimes e_{a_r}\otimes e_{b_r}\otimes e_{c_1}\otimes e_{d_1}\otimes\ldots\otimes e_{c_s}\otimes e_{d_s}
\end{eqnarray*}

Thus, in order to have $g^{\otimes k}\xi_\pi=\xi_\pi$, the matrix $gg^t$ must be a scalar multiple of the identity. Now since this latter condition is not satisfied by any $g\in SU_N$, the formula $g^{\otimes k}\xi_\pi=\xi_\pi$ does not hold in general, and so our partition $\pi$ does not exist, as desired.
\end{proof}

In order to compute now the easiness level of $U_N^d$, we will need some general theory.

We recall from section 1 that each discrete group $\Gamma=<g_1,\ldots,g_N>$ produces a compact quantum group $G=\widehat{\Gamma}$, by setting $C(G)=C^*(\Gamma)$ and $u=diag(g_1,\ldots,g_N)$. The presentation relations $g_{i_1}\ldots g_{i_r}=1$ which define $\Gamma$ correspond then to certain fixed vectors $\xi\in Fix(u^{\otimes r})$. In particular, the presentation level of $\Gamma$, in the group-theoretical sense, is the smallest integer $r\in\mathbb N\cup\{\infty\}$ such that $C=<Fix(u^{\otimes r})>$.

This observation suggests the following definition:

\begin{definition}
Given a closed subgroup $G\subset U_N^+$, with associated Tannakian category $C=(C(k,l))$, the presentation level of its discrete quantum group dual $\Gamma=\widehat{G}$ is the smallest number $r\in\mathbb N\cup\{\infty\}$ such that $C=<Fix(u^{\otimes r})>$.
\end{definition}

As a first observation, in the group dual case, $G=\widehat{\Gamma}$, we obtain indeed the presentation level of $\Gamma$. Indeed, with $u=\sum_ig_i\otimes e_{ii}$ we have $u^{\otimes r}=\sum_{i_1\ldots i_r}g_{i_1}\ldots g_{i_r}\otimes e_{i_1\ldots i_r,i_1\ldots i_r}$, and so a vector $\xi=\sum_{i_1\ldots i_r}\lambda_{i_1\ldots i_r}e_{i_1\ldots i_r}$ is fixed precisely when:
$$\lambda_{i_1\ldots i_r}\neq1\implies g_{i_1}\ldots g_{i_r}=1$$

Thus, just by using vectors $\xi\in(\mathbb C^N)^{\otimes r}$ having $0-1$ entries, we can obtain in this way all the length $r$ relations presenting $\Gamma$. In general now, a matrix $T\in\mathcal L((\mathbb C^N)^{\otimes s},(\mathbb C^N)^{\otimes r})$ belongs to the intertwiner space $Hom(u^{\otimes s},u^{\otimes r})$ precisely we have:
$$T_{i_1\ldots i_r,j_1\ldots j_s}\neq1\implies g_{i_1}\ldots g_{i_r}=g_{j_1}\ldots g_{j_s}$$

Thus, just by using matrices with $0-1$ entries, we obtain the relations presenting $\Gamma$.

Let us recall as well that the Bell numbers $B_r=1,2,5,15,52,\ldots$ count the partitions in $P_r$. These numbers are well-known, but there is no explicit formula for them.

With these conventions, we have the following result:

\begin{proposition}
Consider a homogeneous quantum group $S_N\subset G\subset U_N^+$, and denote by $\Gamma=\widehat{G}$ its discrete quantum group dual.
\begin{enumerate}
\item If $\Gamma$ has presentation level $r<\infty$ then $G$ has easiness level $p\leq B_r$.

\item In particular, if $\Gamma$ is finitely presented, then $G$ has finite easiness level.
\end{enumerate}
\end{proposition}

\begin{proof}
We use the well-known fact, from \cite{wo1}, that we have a Frobenius type isomorphism $Hom(u^{\otimes k},u^{\otimes l})\simeq Fix(u^{\otimes k\bar{l}})$, where $l\to\bar{l}$ is the conjugation of the colored integers. We will use as well the related isomorphism $P(k,l)\simeq P(k\bar{l})$, obtained by rotating.

(1) Let $r=|k|+|l|$ and $p=B_r$, where $k\to|k|$ is the lenght of the colored integers. Since there are exactly $B_r$ elements in $P(k,l)$, we have:
$$\left\{\alpha_1T_{\pi_1}+\ldots+\alpha_pT_{\pi_p}\in C(k,l)\Big|\alpha_i\in\mathbb C,\pi_i\in P(k,l)\right\}=span\left(T_\pi\Big|\pi\in P(k,l)\right)$$

Now since the space on the right includes $Hom(u^{\otimes k},u^{\otimes l})$, we obtain from this, according to the defining formula for the spaces $E^p(k,l)$, from Definition 2.5:
$$E^p(k,l)=Hom(u^{\otimes k},u^{\otimes l})$$

According to our presentation level assumption, this linear space generates $C$. Thus we obtain $C=C^p$, and so we have $G=G^p$, which has easiness level $\leq p$, as desired.

(2) This is clear from (1).
\end{proof}

Now by getting back to our groups $U_N^d$, we obtain:

\begin{theorem}
The easiness level of $U_N^d$ is finite, $l\leq B_{Nd}$.
\end{theorem}

\begin{proof}
We know that $U_N^d\subset U_N$ appears via the relation $(\det g)^d=1$, which reads:
$$\left(\sum_{\sigma\in S_N}\varepsilon(\sigma)g_{1\sigma(1)}\ldots g_{N\sigma(N)}\right)^d=1$$

This relation is a certain linear combination of entries of $g^{\otimes Nd}$, and so corresponds to a certain fixed vector $\xi\in Fix(u^{\otimes Nd})$. Thus, the presentation level of $U_N^d$ is finite, $l\leq Nd$, and the result follows from the general estimate from Proposition 3.5 above.
\end{proof}

The above estimate is of course something quite theoretical. It is probably possible to obtain much better bounds at $N=2$, but we have no results here.

\section{Reflection groups}

We discuss here some further classical examples, this time of discrete nature.

Given a number $s\in\mathbb N\cup\{\infty\}$ we can construct the group $H_N^s=S_N\wr\mathbb Z_s$ of monomial matrices with nonzero entries belonging to the group of $s$-th roots of unity $\mathbb Z_s$, with the usual convention $\mathbb Z_\infty=\mathbb T$. We have the following result, from \cite{fbl}:

\begin{proposition}
The group $H_N^s=S_N\wr\mathbb Z_s$ is easy, with the corresponding category of partitions $P^s$ consisting of the partitions having the property that each block, when weighted according to the rules $\circ\to +,\bullet\to -$, has as size a multiple of $s$. 
\end{proposition}

\begin{proof}
This is something standard, extending some well-known results at $s=1,2$, where $H_N^s$ is respectively the symmetric group $S_N$, and the hyperoctahedral group $H_N$. For full details here, along with a quantum group version of this result, we refer to \cite{fbl}.
\end{proof}

The groups $H_N^s$ are basic examples of complex reflection groups, and belong to the standard series of such groups $\{H_N^{s,d}\}$, depending on an extra parameter $d$. In what follows we will be interested in this series $H_N^{s,d}$, the result that we will need being:

\begin{proposition}
Assuming that $d\in\mathbb N\cup\{\infty\}$ satisfies $2|d|[2,s]$, the group
$$H_N^{s,d}=\left\{g\in H_N^s\Big|\det g\in\mathbb Z_d\right\}$$
is homogeneous. Moreover, when $2|s$ we have in fact $H_N\subset H_N^{s,d}$.
\end{proposition}

\begin{proof}
Observe first that for $g\in H_N^s$ we have $\det g\in\mathbb Z_{[2,s]}$. Now if we assume $d|[2,s]$, in order to avoid redundancy, we have two cases: when $d$ is even we obtain a homogeneous group, as stated, while when $d$ is odd we only have $A_N\subset H_N^{s,d}$. Thus, for obtaining a homogeneous group, we must assume $2|d|[2,s]$. Finally, the last assertion is clear.
\end{proof}

Observe that we have $H_N^{s,d}=H_N^s\cap U_N^d$. In addition, we have a diagram as follows:
$$\xymatrix@R=15mm@C=15mm{
SU_N\ar[r]&U_N^2\ar[r]&U_N^d\ar[r]&U_N\\
A_N\ar[u]\ar[r]&S_N\ar[u]\ar[r]&H_N^d\ar[u]\ar[r]&H_N^\infty\ar[u]
}$$

We already know from Proposition 4.1 that at $d=s$ the group under consideration, namely $H_s^s$ itself, is easy. The case $N=2,s=4,d=2$ is special as well, as follows:

\begin{proposition}
$H_2^{4,2}$ is easy, the corresponding category of partitions being
$$D(k,l)=\begin{cases}
P^2(k,l)&{\rm when}\ \underline{k}=\underline{l}(4)\\
\emptyset&{\rm otherwise}
\end{cases}$$
where $\underline{k}$ is the number $\#\circ-\#\bullet$, over the symbols of $k$.
\end{proposition}

\begin{proof}
According to the definition of $H_N^{s,d}$, we have:
\begin{eqnarray*}
H_2^{4,2}
&=&\left\{g\in H_2^4\Big|\det g\in\mathbb Z_2\right\}\\
&=&\left\{\begin{pmatrix}a&0\\0&b\end{pmatrix},\begin{pmatrix}0&a\\b&0\end{pmatrix}\Big|a,b\in\mathbb Z_4,ab\in\mathbb Z_2\right\}\\
&=&\left\{\begin{pmatrix}a&0\\0&b\end{pmatrix},\begin{pmatrix}0&a\\b&0\end{pmatrix}\Big|a,b=\pm1\ {\rm or}\ a,b=\pm i\right\}\\
&=&H_2\cup iH_2
\end{eqnarray*}

Now observe that by functoriality, and by using as well the result in Proposition 4.1, at $s=2$, the associated Tannakian category $C$ satisfies:
$$C\subset C_{H_2}=span(P^2)$$

In order to compute $C$, we use the trivial fact that the fixed point relations $g^{\otimes l}\xi=\xi$, $(tg)^{\otimes l}\xi=\xi$ with $t\in\mathbb T$ imply $t^l=1$, with the usual conventions $t^\circ=t,t^\bullet=\bar{t}$ for the colored exponents. In our case, with $t=i$ we obtain that we have:
$$C(0,l)\neq\emptyset\implies i^l=1\implies\underline{l}=0(4)$$

More generally, the same method gives in fact the following implications:
$$C(k,l)\neq\emptyset\implies i^k=i^l\implies\underline{k}=\underline{l}(4)$$

We conclude that we have $C\subset span(D)$, where $D=(D(k,l))$ is the collection of sets in the statement. But this collection of sets forms a category of partitions, and by comparing with the classification results in \cite{tw1}, we obtain $C=span(D)$, as stated.
\end{proof}

In what follows, most convenient for the study of $H_N^s$ and its subgroups $H_N^{s,d}$ is to use the wreath product decomposition $H_N^s=S_N\wr\mathbb Z_s$. According to this formula, we have:

\begin{proposition}
Assuming that $d\in\mathbb N\cup\{\infty\}$ satisfies $2|d|[2,s]$, we have
$$H_N^{s,d}=\left\{\sigma(\rho_1,\ldots,\rho_N)\Big|\sigma\in S_N,\rho_i\in\mathbb Z_s,\rho_1\ldots\rho_N\in\mathbb Z_d\right\}$$
where $\sigma(\rho_1,\ldots,\rho_N)=\sum_i\rho_ie_{\sigma(i)i}$, and this group is homogeneous.
\end{proposition}

\begin{proof}
With the convention in the statement for $\sigma(\rho_1,\ldots,\rho_N)$, we have:
$$H_N^s=\left\{\sigma(\rho_1,\ldots,\rho_N)\Big|\sigma\in S_N,\rho_i\in\mathbb Z_s\right\}$$

Consider now an arbitrary number $d\in\mathbb N\cup\{\infty\}$. According to the definition of $H_N^{s,d}$, this group has the following description, where $\varepsilon:S_N\to\{\pm1\}$ is the signature:
$$H_N^{s,d}=\left\{\sigma(\rho_1,\ldots,\rho_N)\Big|\sigma\in S_N,\rho_i\in\mathbb Z_s,\varepsilon(\sigma)\rho_1\ldots\rho_N\in\mathbb Z_d\right\}$$

Now when assuming $2|d$ we have $-1\in\mathbb Z_d$, and so $\varepsilon(\sigma)=\pm1\in\mathbb Z_d$, and we obtain the formula in the statement. As for the homogeneity claim, this is clear as well.
\end{proof}

Regarding now the easy envelope of $H_N^{s,d}$, we have the following result:

\begin{theorem}
We have the easy envelope formula
$$(H_N^{s,d})^1=H_N^s$$
unless we are in the case $(H_2^{4,2})^1=H_2^{4,2}$, which is exceptional.
\end{theorem}

\begin{proof}
We have an inclusion $H_N^{s,d}\subset H_N^s$, and by functoriality, and by using as well the easiness result in Proposition 4.1 above, we succesively obtain:
$$H_N^{s,d}\subset H_N^s\implies span(P^s)\subset C\implies P^s\subset D^1$$

In order to prove the reverse inclusion $D^1\subset P^s$, we must compute $D^1$. 

By using Proposition 2.6, it is enough to discuss the fixed points. For a partition $\pi\in P(k)$, the associated vector $T_\pi$, that we will denote here by $\xi_\pi$, is given by:
$$\xi_\pi=\sum_{i_1\ldots i_k}\delta_\pi(i_1,\ldots,i_k)e_{i_1}\otimes\ldots\otimes e_{i_k}$$

Now with $g=\sigma(\rho_1,\ldots,\rho_N)\in H_N^{s,d}$, as in Proposition 4.4, we have:
$$g^{\otimes k}\xi_\pi=\sum_{i_1\ldots i_k}\delta_\pi(i_1,\ldots,i_k)\rho_{i_1}\ldots\rho_{i_k}\ e_{i_{\sigma(1)}}\otimes\ldots\otimes e_{i_{\sigma(k)}}$$

On the other hand, by replacing $i_r\to i_{\sigma(r)}$, we have as well:
\begin{eqnarray*}
\xi_\pi
&=&\sum_{i_1\ldots i_k}\delta_\pi(i_{\sigma(1)},\ldots,i_{\sigma(k)})\ e_{i_{\sigma(1)}}\otimes\ldots\otimes e_{i_{\sigma(k)}}\\
&=&\sum_{i_1\ldots i_k}\delta_\pi(i_1,\ldots,i_k)\ e_{i_{\sigma(1)}}\otimes\ldots\otimes e_{i_{\sigma(k)}}
\end{eqnarray*}

We conclude from this that the formula $g^{\otimes k}\xi_\pi=\xi_\pi$ is equivalent to:
$$\delta_\pi(i_1,\ldots,i_k)=1\implies\rho_{i_1}\ldots\rho_{i_k}=1$$

To be more precise, in order for $g^{\otimes k}\xi_\pi=\xi_\pi$ to hold, this formula must hold for any numbers $\rho_1,\ldots,\rho_N\in\mathbb Z_s$ satisfying $\rho_1\ldots\rho_N\in\mathbb Z_d$. 

Observe that in the case $d=s$ the condition $\rho_1\ldots\rho_N\in\mathbb Z_d$ dissapears, and the condition $\delta_\pi(i_1,\ldots,i_k)=1\implies\rho_{i_1}\ldots\rho_{i_k}=1$, for any $\rho_1,\ldots,\rho_N\in\mathbb Z_s$, tells us that all the blocks of $\pi$, when weighted according to the rules $\circ\to +,\bullet\to -$, must have as size a multiple of $s$. Thus $\pi\in P^s$. This is something that we already know, from Proposition 4.1.

Now back to our question, so far we have obtained:
$$D_1(k)=\left\{\pi\Big|\delta_\pi(i_1,\ldots,i_k)=1\implies\rho_{i_1}\ldots\rho_{i_k}=1,\forall\rho_1,\ldots,\rho_N\in\mathbb Z_s,\rho_1\ldots\rho_N\in\mathbb Z_d\right\}$$

In order to compute this set, let $\pi$ and $i_1,\ldots,i_k$ be as above, and consider the partition $\nu=\ker i$. We have then $\nu\leq\pi$, and since $i_1,\ldots,i_k\in\{1,\ldots,N\}$, we have $r\leq N$. 

Depending now on the value of $r=|\nu|$, we have two cases, as follows:

(1) In the case $N>r$ we have a free variable among $\{\rho_1,\ldots,\rho_N\}$, that we can adjust as to have $\rho_1\ldots\rho_N\in\mathbb Z_d$. Thus, the condition $\rho_1\ldots\rho_N\in\mathbb Z_d$ dissapears, and we are left with the $H_N^s$ problem, which gives, as explained above, $\nu\in P_s$.

(2) In the case $N=r$, let us denote by $a_1+b_1,\ldots,a_N+b_N$ the lengths of the blocks of $\nu$, with $a_i$ standing for the white legs, and $b_i$ standing for the black legs. We have:
$$\rho_1^{a_1-b_1}\ldots\rho_N^{a_N-b_N}=1,\forall\rho_1,\ldots,\rho_N\in\mathbb Z_s,\rho_1\ldots\rho_N\in\mathbb Z_d$$

With $c_i=a_i-b_i$, and with $\eta_N=\rho_1\ldots\rho_N$, we must have:
$$\rho_1^{c_1-c_N}\ldots\rho_{N-1}^{c_{N-1}-c_N}\eta_N^{c_N}=1,\forall\rho_1,\ldots,\rho_{N-1}\in\mathbb Z_s,\forall \eta_N\in\mathbb Z_d$$

Thus we must have $c_1=\ldots=c_N(s)$, and this common value must be a number $c=0(d)$. Now let us introduce the following sets:
$$P_c^{s,d}=\left\{\pi\Big||\pi|=N,a_i-b_i=c(s)\right\}$$

In terms of these sets, and of their union $P^{s,d}=\cup_cP_c^{s,d}$, we have obtained that $\pi\in D^1$ happens if and only if any subpartition $\nu\leq\pi$ has the following property:

(1) If $|\nu|<N$, then $\nu\in P^s$.

(2) If $|\nu|=N$, then $\nu\in P^{s,d}$.

(3) If $|\nu|>N$, no condition.

But this shows that we must have $\pi\in P^s$, unless we are in the exceptional case, $N=2,s=4,d=2$. Thus we have $(H_N^{s,d})^1=H_N^s$, as stated.
\end{proof}

Finally, regarding the easiness level of $H_N^{s,d}$, we have:

\begin{proposition}
The easiness level of $H_N^{s,d}$ is finite, $p\leq B_{Nd}$.
\end{proposition}

\begin{proof}
This follows from Proposition 3.5 above, because the formula $(\det g)^d=1$ presents the corresponding tensor category, and discrete quantum group. To be more precise, $H_N^{s,d}$ has presentation level $r\leq Nd$, and so has easiness level $p\leq B_{Nd}$.
\end{proof}

The above result is of course quite theoretical, the upper bound for $p$ found there being too big. Some more precise results can be probably obtained at $N=2$.

\section{Half-liberations}

In this section and in the next one we work out some more examples, this time in the non-classical setting. There are not many non-easy candidates here, but one interesting construction comes from the bistochastic groups and quantum groups $B_N\subset B_N^+$, and the orthogonal quantum groups $O_N\subset O_N^*\subset O_N^+$ from \cite{bsp}.

We agree to use matrix indices $i,j=0,1,\ldots,N-1$ for the $N\times N$ compact matrix quantum groups, and indices $i,j=1,\ldots,N-1$ for their $(N-1)\times(N-1)$ subgroups. With this convention, we have the following result, from \cite{rau}:

\begin{proposition}
We have an isomorphism $B_N^+\simeq O_{N-1}^+$, whose transpose is given by
$$C(O_{N-1}^+)\to C(B_N^+)\quad,\quad w_{ij}\to(F^*uF)_{ij}$$
whenever $F\in O_N$ satisfies $Fe_0=\frac{1}{\sqrt{N}}\xi$, where $\xi$ is the all-one vector. 
\end{proposition}

\begin{proof}
Assuming $Fe_0=\frac{1}{\sqrt{N}}\xi$ as above, we have the following computation:
\begin{eqnarray*}
u\xi=\xi
&\iff&uFe_0=Fe_0\\
&\iff&F^*uFe_0=e_0\\
&\iff&F^*uF=diag(1,w)
\end{eqnarray*}

But this gives the isomorphism in the statement. See \cite{rau}.
\end{proof}

We can therefore construct intermediate objects for $B_N\subset B_N^+$, as follows:

\begin{proposition}
Assuming that $F\in O_N$ satisfies $Fe_0=\frac{1}{\sqrt{N}}\xi$, we have inclusions as follows, with the intermediate quantum group $B_F^\circ$ being not easy,
$$B_N\subset B_F^\circ\subset B_N^+$$
obtained by taking the image of the inclusions $O_{N-1}\subset O_{N-1}^*\subset O_{N-1}^+$, via the above isomorphism $O_{N-1}^+\simeq B_N^+$ induced by $F$.
\end{proposition}

\begin{proof}
The fact that we have inclusions as in the statement follows from Proposition 5.1 above, which produces a diagram as follows:
$$\xymatrix@R=15mm@C=20mm{
O_{N-1}\ar[r]&O_{N-1}^*\ar[r]&O_{N-1}^+\\
B_N\ar@.[r]\ar@{=}[u]&B_F^\circ\ar@.[r]\ar@{=}[u]&B_N^+\ar@{=}[u]
}$$

To be more precise, the quantum group $B_F^\circ$ from the bottom is by definition the image of the quantum group $O_{N-1}^*$ from the top. Since we know that $B_N\subset B_N^+$ is maximal in the easy setting, this new quantum group $B_F^\circ$ is not easy, as claimed.
\end{proof}

Let us record as well the following result:

\begin{proposition}
The quantum group $B_F^\circ\subset B_N^+$  appears via the relations
$$abc=cba\quad,\quad\forall a,b,c\in\left\{(F^*uF)_{ij}\Big|i,j=1,\ldots,N-1\right\}$$
and its presentation level is $r=6$. 
\end{proposition}

\begin{proof}
The first assertion is clear from definitions. By using now the Frobenius isomorphism $End(u^{\otimes 3})\simeq Fix(u^{\otimes 6})$ we conclude that the level is $\leq6$, and the point is that the level is precisely 6, by using the isomorphism with $O_{N-1}^*$, which is of level 6.
\end{proof}

Observe that the relations $abc=cba$ do not hold for all the entries of the modified fundamental corepresentation $v=F^*uF$, due to the fact that we have $v_{00}=1$, and that the relations $ab1=1ba$ corresponds to the commutativity. We have in fact:

\begin{proposition}
The quantum group $B_F^\circ\subset B_N^+$  appears via the relations
$$R^{\otimes 3}T_{\slash\hskip-1.2mm|\hskip-1.2mm\backslash}R^{*\otimes 3}\in End(u^{\otimes 3})$$
where $R=FP$, with $R$ being the projection onto $span(e_1,\ldots,e_{N-1})$.
\end{proposition}

\begin{proof}
With $F^*uF=diag(1,w)$, as in the proof of Proposition 5.1, the relations are:
\begin{eqnarray*}
T_{\slash\hskip-1.2mm|\hskip-1.2mm\backslash}\in End(w^{\otimes 3})
&\iff&P^{\otimes 3}T_{\slash\hskip-1.2mm|\hskip-1.2mm\backslash}P^{\otimes 3}\in End((F^*uF)^{\otimes 3})\\
&\iff&(FP)^{\otimes 3}T_{\slash\hskip-1.2mm|\hskip-1.2mm\backslash}(FP)^{*\otimes 3}\in End(u^{\otimes 3})
\end{eqnarray*}

Thus, we obtain the formula in the statement.
\end{proof}

Now observe that, due to the conditions $F\in O_N$ and $Fe_0=\frac{1}{\sqrt{N}}\xi$, the linear map associated to $R=FP$ maps $e_0\to0\to0$ and $e_i\to e_i\to f_i$, where $\{f_1,\ldots,f_{N-1}\}$ is a certain orthonormal basis of $\xi^\perp$. Thus $R=FP$ must be a partial isometry $e_0^\perp\to\xi^\perp$.

We can further process the above result, as follows:

\begin{proposition}
The quantum groups $B_F^\circ\subset B_N^+$ with $F\in O_N$, $Fe_0=\frac{1}{\sqrt{N}}\xi$ all coincide, and appear via the relations $T\in End(u^{\otimes 3})$, where
\begin{eqnarray*}
T(e_i\otimes e_j\otimes e_k)
&=&e_k\otimes e_j\otimes e_i-(e_k\otimes e_j\otimes\xi'+e_k\otimes\xi'\otimes e_i+\xi'\otimes e_j\otimes e_i)\\
&&+(e_k\otimes\xi'\otimes\xi'+\xi'\otimes e_j\otimes\xi'+\xi'\otimes\xi'\otimes e_i)-\xi'\otimes\xi'\otimes\xi'
\end{eqnarray*}
with $\xi'=\frac{1}{N}\xi$, and with $\xi$ being as usual the all-one vector.
\end{proposition}

\begin{proof}
The linear map $R^{\otimes 3}T_{\slash\hskip-1.2mm|\hskip-1.2mm\backslash}R^{*\otimes 3}$ from Proposition 5.4 acts as follows:
\begin{eqnarray*}
&&R^{\otimes 3}T_{\slash\hskip-1.2mm|\hskip-1.2mm\backslash}R^{*\otimes 3}(e_i\otimes e_j\otimes e_k)\\
&=&R^{\otimes 3}\sum_{abc}R_{ia}R_{jb}R_{kc}\ e_c\otimes e_b\otimes e_a\\
&=&\sum_{abcpqr}R_{ia}R_{jb}R_{kc}R_{pc}R_{qb}R_{ca}\ e_p\otimes e_q\otimes e_r\\
&=&\sum_{pqr}(RR^t)_{ir}(RR^t)_{jq}(RR^t)_{kp}\ e_p\otimes e_q\otimes e_r
\end{eqnarray*}

On the other hand, since $R=FP$ must be a partial isometry $e_0^\perp\to\xi^\perp$, we have:
$$RR^*=1-Proj(\xi)\quad,\quad(RR^*)_{ij}=\delta_{ij}-\frac{1}{N}$$

We conclude that $R^{\otimes 3}T_{\slash\hskip-1.2mm|\hskip-1.2mm\backslash}R^{*\otimes 3}$ is given by:
$$e_i\otimes e_j\otimes e_k\to\sum_{pqr}\left(\delta_{ir}-\frac{1}{N}\right)\left(\delta_{jq}-\frac{1}{N}\right)\left(\delta_{kp}-\frac{1}{N}\right)\ e_p\otimes e_q\otimes e_r$$

By developing, we obtain the formula in the statement.
\end{proof}

An even better statement is as follows:

\begin{proposition}
The quantum group $B_N^\circ\subset B_N^+$ constructed above, which equals the various quantum groups $B_F^\circ$, appears via the relations $T\in End(u^{\otimes 3})$, where
$$T=T_{\slash\hskip-1.2mm|\hskip-1.2mm\backslash}-
(T_{{\ }^\cdot_\cdot\slash\hskip-1.4mm\backslash}+T_{{\ }^{{\ }^\cdot}_{{\ }_\cdot}\hskip-1.5mm\slash\hskip-1.4mm\backslash}+T_{\slash\hskip-1.4mm\backslash\!\!\!{\ }^\cdot_\cdot})+
(T_{{\ }_{..}\!\!\backslash^{\!\cdot\cdot}}+T_{{\ }^\cdot_\cdot|\!\!\!{\ }^\cdot_\cdot}+T_{{\ }^{..}\!\slash_{\!\cdot\cdot}})-
T_{{\ }^{\cdot\cdot\cdot}_{\cdot\cdot\cdot}}$$
with the convention that the various dots represent singletons.
\end{proposition}

\begin{proof}
This follows indeed from the formula in Proposition 5.5 above, because the 8 terms there correspond to the 8 partitions in the statement. 
\end{proof}

Observe that we can in fact write an even more compact formula, as follows:
$$T=\sum_{\pi\leq\slash\hskip-1.2mm|\hskip-1.2mm\backslash}\mu(\pi)T_\pi$$

To be more precise, here the sum is over all the partitions $\pi\in P_{12}(3,3)$ satisfying $\pi\leq\slash\hskip-1.6mm|\hskip-1.6mm\backslash$, and the numbers $\mu(\pi)\in\{\pm1\}$ come from the M\"obius function of $P_{12}$, where $P_{12}$ is the category of singletons and pairings, known from \cite{bsp} to produce $B_N$.

Now by getting back to the notion of easiness level, we have:

\begin{theorem}
The quantum group $B_N^\circ\subset B_N^+$ has the following properties:
\begin{enumerate}
\item Its easy envelope is $B_N^+$.

\item Its presentation level is $r=6$.

\item Its easiness level is $p\leq 8$.
\end{enumerate}
\end{theorem}

\begin{proof}
We use the various results established above.

(1) This is clear from the fact that $B_N\subset B_N^+$ is maximal, in the easy setting.

(2) This is something that we already know, from Proposition 5.3 above.

(3) Observe first that (2) and Proposition 3.5 give $p\leq B_6=203$. However, by using Proposition 5.6 above we obtain the finer estimate $p\leq8$, as stated.
\end{proof}

\section{Unitary versions}

In this section we work out the unitary versions of the constructions from the previous section. We recall from \cite{tw1} that the complex bistochastic group $C_N$ consists by definition of the matrices $U\in U_N$ having sum 1 on each row and each column. Its free analogue $C_N^+$ can be constructed from $U_N^+$ by using the relation $u\xi=\xi$. See \cite{tw1}.

In analogy with Proposition 5.1, we have the following result:

\begin{proposition}
We have an isomorphism $C_N^+\simeq U_{N-1}^+$, whose transpose is given by
$$C(U_{N-1}^+)\to C(C_N^+)\quad,\quad w_{ij}\to(F^*uF)_{ij}$$
whenever $F\in U_N$ satisfies $Fe_0=\frac{1}{\sqrt{N}}\xi$, where $\xi$ is the all-one vector. 
\end{proposition}

\begin{proof}
Assuming $Fe_0=\frac{1}{\sqrt{N}}\xi$ as above, we have the following computation:
\begin{eqnarray*}
u\xi=\xi
&\iff&uFe_0=Fe_0\\
&\iff&F^*uFe_0=e_0\\
&\iff&F^*uF=diag(1,w)
\end{eqnarray*}

But this gives the isomorphism in the statement. See \cite{rau}.
\end{proof}

As a first remark, the situation in the unitary case is slightly different, coming from the fact that we have several examples of intermediate easy quantum groups $U_N\subset G\subset U_N^+$. Such quantum groups are in fact far from being classified. See \cite{ba1}, \cite{bb2}, \cite{gro}, \cite{tw1}.

The most basic example of an intermediate easy quantum group  $U_N\subset G\subset U_N^+$ is the quantum group $U_N^*$ from \cite{bdu}, which appears from $U_N^+$ via the half-commutation relations relations $abc=cba$, imposed to the standard coordinates $u_{ij}$, and their adjoints $u_{ij}^*$.

By proceeding as in section 4 above, we obtain:

\begin{proposition}
The image of the intermediate quantum group $U_{N-1}\subset U_{N-1}^*\subset U_{N-1}^+$ via the above isomorphism $U_{N-1}^+\simeq C_N^+$ induced by $F$ is an intermediate quantum group $C_N\subset C_N^\circ\subset C_N^+$, not depending on $F$. Moreover, $C_N^\circ\subset C_N^+$
appears via the relations
$$T\in End(u^{\otimes k})$$
where $T$ is the linear map from Proposition 5.6, and where $k\in\{\circ\circ\circ,\circ\circ\bullet,\ldots,\bullet\bullet\bullet\}$ ranges over all the colored integers of length $3$.
\end{proposition}

\begin{proof}
This follows indeed by proceeding as in section 4 above, and by replacing where needed the tensor powers $u^{\otimes 3}$ by the colored tensor powers $u^{\otimes k}$, as above.
\end{proof}

Another interesting example of an intermediate easy quantum group $U_N\subset G\subset U_N^+$ is the quantum group $U_N^*\subset U_N^\times\subset U_N^+$ constructed in \cite{bdd}, which appears via the relations $ab^*c=cb^*a$, imposed to the standard coordinates $u_{ij}$. We have here:

\begin{proposition}
The image of the intermediate quantum group $U_{N-1}^*\subset U_{N-1}^\times\subset U_{N-1}^+$ via the above isomorphism $U_{N-1}^+\simeq C_N^+$ induced by $F$ is an intermediate quantum group $C_N^\circ\subset C_N^\times\subset C_N^+$, not depending on $F$. Moreover, $C_N^\times\subset C_N^+$ appears via the relations
$$T\in End(u\otimes\bar{u}\otimes u)$$
where $T$ is the linear map from Proposition 5.6.
\end{proposition}

\begin{proof}
Once again, this follows by proceeding as in section 4 above.
\end{proof}

Finally, one example of a slightly different nature is the intermediate quantum group $U_N^*\subset U_N^{**}\subset U_N^\times$ constructed in \cite{bb2}, which appears via the commutation relations between all variables $\{ab^*,a^*b\}$, with $a,b$ ranging over the standard coordinates $u_{ij}$. 

As explained in \cite{bb2}, we have the following result:

\begin{proposition}
The quantum group $U_N^{**}$ is easy, the corresponding category of partitions being generated by the following diagrams:
$$
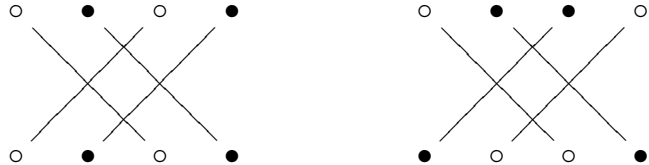
\xymatrix@R=15mm@C=5mm{\circ\ar@{-}[drr]&\bullet\ar@{-}[drr]&\circ\ar@{-}[dll]&\bullet\ar@{-}[dll]\\\circ&\bullet&\circ&\bullet}
\qquad \quad\qquad 
\xymatrix@R=15mm@C=5mm{\circ\ar@{-}[drr]&\bullet\ar@{-}[drr]&\bullet\ar@{-}[dll]&\circ\ar@{-}[dll]\\\bullet&\circ&\circ&\bullet}$$
In addition, $U_N^{**}$ is the biggest subgroup of $U_N^+$ having the property that its full projective version, having as coordinates the entries of $u\otimes\bar{u}+\bar{u}\otimes u$, is classical.
\end{proposition}

\begin{proof}
Both the assertions follow from definitions. For more details on these facts, and for some interpretations of the last assertion, we refer to \cite{bb2}.
\end{proof}

We can now construct one more bistochastic quantum group, as follows:

\begin{proposition}
The image of the intermediate quantum group $U_{N-1}^*\subset U_{N-1}^{**}\subset U_{N-1}^\times$, via the above isomorphism $U_{N-1}^+\simeq C_N^+$ induced by $F$, is an intermediate quantum group $C_N^\circ\subset C_N^{\circ\circ}\subset C_N^\times$, not depending on $F$. Moreover, $C_N^{\circ\circ}\subset C_N^+$ appears via the relations
$$T\in Hom(u^{\otimes\circ\bullet\circ\bullet},u^{\otimes\circ\bullet\circ\bullet})\quad,\quad T\in Hom(u^{\otimes\circ\bullet\bullet\circ},u^{\otimes\bullet\circ\circ\bullet})$$
where $T=\sum_{\pi\leq\slash\hskip-0.9mm\slash\hskip-2.1mm\backslash\hskip-0.9mm\backslash}\mu(\pi)T_\pi$, with the sum taken inside $\mathcal P_{12}$, the category of matching singletons and pairings, and where $\mu(\pi)=\pm1$ are M\"obius function signs.
\end{proposition}

\begin{proof}
This follows by using the same method as in section 5 above, with the main computation, which is analogous to the one in the proof of Proposition 5.5, being:
\begin{eqnarray*}
&&R^{\otimes 4}T_{\slash\hskip-0.9mm\slash\hskip-2.1mm\backslash\hskip-0.9mm\backslash}R^{*\otimes 4}(e_i\otimes e_j\otimes e_k\otimes e_l)\\
&=&R^{\otimes 4}T_{\slash\hskip-0.9mm\slash\hskip-2.1mm\backslash\hskip-0.9mm\backslash}\sum_{abcd}\bar{R}_{ia}\bar{R}_{jb}\bar{R}_{kc}\bar{R}_{ld}\ e_a\otimes e_b\otimes e_c\otimes e_d\\
&=&R^{\otimes 4}\sum_{abcd}\bar{R}_{ia}\bar{R}_{jb}\bar{R}_{kc}\bar{R}_{ld}\ e_c\otimes e_d\otimes e_a\otimes e_b\\
&=&\sum_{abcdpqrs}\bar{R}_{ia}\bar{R}_{jb}\bar{R}_{kc}\bar{R}_{ld}R_{pc}R_{qd}R_{ra}R_{sb}\ e_p\otimes e_q\otimes e_r\otimes e_s\\
&=&\sum_{pqrs}(RR^*)_{ri}(RR^*)_{sj}(RR^*)_{pk}(RR^*)_{ql}\ e_p\otimes e_q\otimes e_r\otimes e_s\\
&=&\sum_{pqrs}\left(\delta_{ir}-\frac{1}{N}\right)\left(\delta_{js}-\frac{1}{N}\right)\left(\delta_{kp}-\frac{1}{N}\right)\left(\delta_{lq}-\frac{1}{N}\right)\ e_p\otimes e_q\otimes e_r\otimes e_s
\end{eqnarray*}

Indeed, this linear map is the one in the statement, and this gives the result.
\end{proof}

Now by getting back to the notion of easiness level, we have:

\begin{theorem}
The quantum groups $C_N^\circ\subset C_N^{\circ\circ}\subset C_N^\times$ have the following properties:
\begin{enumerate}
\item They have a common easy envelope, namely $C_N^+$.

\item Their presentation level is $r\leq 6,6,8$, respectively.

\item Their easiness level is $p\leq8,8,16$, respectively.
\end{enumerate}
\end{theorem}

\begin{proof}
We use the various results established above:

(1) Our claim here is that the inclusion $C_N\subset C_N^+$ is maximal, in the easy setting. Indeed, given an intermediate category ${\mathcal NC}_{12}\subset D\subset\mathcal P_{12}$, assuming $D\neq {\mathcal NC}_{12}$ we can select a crossing partition $\pi\in D$, and then cap to singletons, as to reach to the conclusion that we must have $\slash\hskip-2.1mm\backslash\in D$, and so $D=\mathcal P_{12}$. Thus $C_N\subset C_N^+$ is indeed maximal, in the easy setting, and so the easy envelope of any intermediate quantum group must equal $C_N^+$.

(2) This follows from Propositions 6.2, 6.3, 6.5 above.

(3) This follows as well by using Propositions 6.2, 6.3, 6.5 above.
\end{proof}

\section{Maximality questions}

In this section and in the next one we discuss various maximality questions, in the orthogonal case. The idea is that in the orthogonal case, where the easy quantum groups are fully classified \cite{rw3}, there are a number of inclusions $G_N\subset G_N^\times$, known to be maximal in the easy setting, in the sense that there is no intermediate easy quantum group $G_N\subset G\subset G_N^\times$. The problem of investigating the maximality of these inclusions in the arbitrary compact quantum group setting appears, and we will discuss here this question.

We will split our study into two parts, with the present section containing preliminaries and some technical results. In order to get started, we will need:

\begin{proposition}
For a liberation operation of easy quantum groups $G_N\to G_N^+$, the following conditions are equivalent:
\begin{enumerate}
\item The category $\mathcal P_2\subset D\subset P$ associated to $G=(G_N)$, or, equivalently, the category ${\mathcal NC}_2\subset D\subset NC$ associated to $G^+=(G_N^+)$, is stable under removing blocks.

\item We have $G_N\cap U_K=G_K$, or, equivalently, $G_N^+\cap U_K^+=G_K^+$, for any $K\leq N$, where the embeddings $U_K\subset U_N$ and $U_K^+\subset U_N^+$ are the standard ones.

\item Each $G_N$ appears as lift of its projective version $G_N\to PG_N$, or, equivalently, each $G_N^+$ appears as lift of its projective version $G_N^+\to PG_N^+$.

\item The laws of truncated characters $\chi_t=\sum_{i=1}^{[tN]}u_{ii}$, with $t\in(0,1]$, for $G_N$ and $G_N^+$, form convolution/free convolution semigroups, in Bercovici-Pata bijection.
\end{enumerate}
If these conditions are satisfied, we call $G_N\to G_N^+$ a ``true'' liberation.
\end{proposition}

\begin{proof}
All this is well-known, basically going back to \cite{bsp}, the idea being that the implications $(1)\iff(2)\iff(3)$ are all elementary, and that $(1)\iff(4)$ follows by using the cumulant interpretation of the Bercovici-Pata bijection \cite{bpa}, stating that ``the classical cumulants become via the bijection free cumulants''. See \cite{bsp}, \cite{nsp}.
\end{proof}

The above result is of course something rather theoretical. In practice, we will now restrict the attention to the orthogonal case. By using \cite{rw3}, we obtain:

\begin{proposition}
There are precisely $4$ true liberations of orthogonal easy quantum groups, with the intermediate liberations, in the easy framework, as follows:
\begin{enumerate}
\item $S_N\subset S_N^+$, with no intermediate object.

\item $H_N\subset H_N^+$, with uncountably many intermediate objects.

\item $O_N\subset O_N^+$, with $O_N^*$ as unique intermediate object.

\item $B_N\subset B_N^+$, with no intermediate object.
\end{enumerate}
\end{proposition}

\begin{proof}
This follows from the classification results in \cite{rw3}, by using Proposition 7.1 above. For details here, we refer to the lecture notes \cite{ba3}.
\end{proof}

Summarizing, we have 4 inclusions to look at. The inclusion $S_N\subset S_N^+$ is conjectured to be maximal, and we will discuss this a bit later, at the end of this section. The situation with $H_N\subset H_N^+$ is too complex, and we will remove this inclusion from our study. Thus, we are left with the ``continuous'' inclusions, $O_N\subset O_N^+$ and $B_N\subset B_N^+$.

According to \cite{max}, and by using the intermediate quantum group $B_N\subset B_N^\circ\subset B_N^+$ constructed in section 4 above, we have the following result:

\begin{proposition}
The following inclusions are maximal:
\begin{enumerate}
\item $O_N\subset O_N^*$.

\item $B_N\subset B_N^\circ$.
\end{enumerate}
\end{proposition}

\begin{proof}
Here (1) is from \cite{max}, and (2) follows from it, by functoriality. As a remark here, the passage from the maximality of the inclusion $PO_N\subset PU_N$, proved in \cite{max}, to the maximality of $O_N\subset O_N^*$ itself, follows now as well directly, by using the conceptual approach to the half-liberation operation worked out in \cite{bdu}. To be more precise, it follows from \cite{bdu} that an intermediate proper quantum group $O_N\subset G\subset O_N^*$ must come from a certain proper subgroup $H\subset U_N$, satisfying $PH=PG$, and this gives the result.
\end{proof}

It was conjectured in \cite{max} that $O_N\subset O_N^*\subset O_N^+$ is  ``complete'', in the sense that $O_N^*$ should be the unique intermediate quantum group $O_N\subset G\subset O_N^+$, and our conjecture here would be that $B_N\subset B_N^\circ\subset B_N^+$ should be complete as well.

Let us discuss now the inclusion $S_N\subset S_N^+$. This is perhaps the most important inclusion from our list in Proposition 7.2, and an old and well-known conjecture, going back to \cite{bb1}, states that this inclusion is maximal, in the general quantum group setting.

At $N=1,2,3$ this holds indeed, because here, as explained in \cite{wa2}, we have $S_N=S_N^+$. At $N=4$ now, we have the following non-trivial result, from \cite{bb1}:

\begin{proposition}
We have $S_4^+=SO_3^{-1}$, with isomorphism given by the Fourier transform over the Klein group $K=\mathbb Z_2\times\mathbb Z_2$, and the subgroups of this quantum group are:
\begin{enumerate}
\item Infinite quantum groups: $S_4^+$, $O_2^{-1}$, $\widehat{D}_\infty$.

\item Finite groups: $S_4$, and its subgroups.

\item Finite group twists: $S_4^{-1}$, $A_5^{-1}$.

\item Series of twists: $D_{2n}^{-1}$ $(n\geq 3)$, $DC^{-1}_n$ $(n\geq 2)$.

\item A group dual series: $\widehat{D}_n$, with $n\geq 3$.
\end{enumerate}
In particular, these quantum groups are subject to an ADE classification result. Also, the inclusion $S_4\subset S_4^+$ follows to be maximal.
\end{proposition}

\begin{proof}
All this is quite technical, the idea being that the classification result can be obtained by taking some inspiration from the McKay classification of the subgroups of $SO_3$. As for the last assertion, this follows from the classification. See \cite{bb1}.
\end{proof}

We will prove in what follows that the inclusion $S_5\subset S_5^+$ is maximal as well, by using as main input some recent advances in subfactor theory, from \cite{imp}. 

Let us first study the arbitrary quantum subgroups $G\subset S_5^+$. As a first, elementary observation here, we have:

\begin{proposition}
We have the following examples of subgroups $G\subset S_5^+$:
\begin{enumerate}
\item The classical subgroups, $G\subset S_5$. There are $16$ such subgroups, having order $1,2,3,4,4,5,6,6,8,10,12,12,20,24,60,120$.

\item The group duals, $G=\widehat{\Gamma}\subset S_5^+$. These appear, via a Fourier transform construction, from the various quotients $\Gamma$ of the groups $\mathbb Z_4,\mathbb Z_2*\mathbb Z_2,\mathbb Z_2*\mathbb Z_3$.
\end{enumerate}
In addition, we have as well all the ADE quantum groups $G\subset S_4^+\subset S_5^+$, embedded via the $5$ standard embeddings $S_4^+\subset S_5^+$.
\end{proposition}

\begin{proof}
These results are well-known, the proof being as follows:

(1) This is a classical result, with the groups which appear being respectively the cyclic groups $\{1\},\mathbb Z_2,\mathbb Z_3,\mathbb Z_4$, the Klein group $K=\mathbb Z_2\times\mathbb Z_2$, then $\mathbb Z_5,\mathbb Z_6,S_3,D_4,D_5,A_4$, then a copy of $S_3\rtimes\mathbb Z_2$, the general affine group $GA_1(5)=\mathbb Z_5\rtimes\mathbb Z_4$, and finally $S_4,A_5,S_5$.

(2) This follows from Bichon's result in \cite{bic}, stating that the group dual subgroups $G=\widehat{\Gamma}\subset S_N^+$ appear from the various quotients $\mathbb Z_{N_1}*\ldots*\mathbb Z_{N_k}\to\Gamma$, with $N_1+\ldots+N_k=N$. At $N=5$ the partitions are $5=1+4,1+2+2,2+3$, and this gives the result.
\end{proof}

Summarizing, the classification of the subgroups $G\subset S_5^+$ looks like a particularly difficult task, the situation here being definitely much more complicated than what happens at $N=4$. We will restrict the attention in what follows to the transitive case:

\begin{definition}
Let $G\subset S_N^+$ be a closed subgroup, with magic unitary $u=(u_{ij})$, and consider the equivalence relation on $\{1,\ldots,N\}$ given by $i\sim j\iff u_{ij}\neq0$.
\begin{enumerate}
\item The equivalence classes under $\sim$ are called orbits of $G$.

\item $G$ is called transitive when the action has a single orbit. 
\end{enumerate}
In other words, we call a subgroup $G\subset S_N^+$ transitive when $u_{ij}\neq0$, for any $i,j$.
\end{definition}

This transitivity notion is standard, coming from Bichon's work \cite{bic}. In relation to our questions, observe that any intermediate quantum group $S_N\subset G\subset S_N^+$ must be transitive. Thus, we can restrict the attention to such quantum groups. We have:

\begin{proposition}
We have the following examples of transitive subgroups $G\subset S_5^+$:
\begin{enumerate}
\item The classical transitive subgroups $G\subset S_5$. There are only $5$ such subgroups, namely $\mathbb Z_5,D_5,GA_1(5),A_5,S_5$.

\item The transitive group duals, $G=\widehat{\Gamma}\subset S_5^+$. There is only one example here, namely the dual of $\Gamma=\mathbb Z_5$, which is $\mathbb Z_5$, already appearing above.
\end{enumerate}
In addition, all the ADE quantum groups $G\subset S_4^+\subset S_5^+$ are not transitive.
\end{proposition}

\begin{proof}
This follows indeed by examining the lists in Proposition 7.5:

(1) The result here is well-known, and elementary. Observe that $GA_1(5)=\mathbb Z_5\rtimes\mathbb Z_4$, which is by definition the general affine group of $\mathbb F_5$, is indeed transitive.

(2) This follows from the results in \cite{bic}, because with $\mathbb Z_{N_1}*\ldots*\mathbb Z_{N_k}\to\Gamma$ as in the proof of Proposition 7.5 (2), the orbit decomposition is precisely $N=N_1+\ldots+N_k$.

Finally, the last assertion is clear, because the embedding $S_4^+\subset S_5^+$ is obtained precisely by fixing a point. Thus $S_4^+$ and its subgroups are not transitive, as claimed.
\end{proof}

In order to prove the uniqueness result, we will use the recent progress in subfactor theory \cite{imp}. For our purposes, the most convenient formulation of the result in \cite{imp} is:

\begin{proposition}
The principal graphs of the irreducible index $5$ subfactors are:
\begin{enumerate}
\item $A_\infty$, and a non-extremal perturbation of $A_\infty^{(1)}$.

\item The McKay graphs of $\mathbb Z_5,D_5,GA_1(5),A_5,S_5$.

\item The twists of the McKay graphs of $A_5,S_5$.
\end{enumerate}
\end{proposition}

\begin{proof}
This is a heavy result, and we refer to \cite{imp}, \cite{jms} for the whole story. The above formulation is the one from \cite{imp}, with the subgroup subfactors there replaced by fixed point subfactors \cite{ba2}, and with the cyclic groups denoted as usual by $\mathbb Z_N$.
\end{proof}

In the quantum permutation group setting, this result becomes:

\begin{proposition}
The set of principal graphs of the transitive subgroups $G\subset S_5^+$ coincide with the set of principal graphs of the subgroups $\mathbb Z_5,D_5,GA_1(5),A_5,S_5,S_5^+$.
\end{proposition}

\begin{proof}
We must take the list of graphs in Proposition 7.8, and exclude some of the graphs, on the grounds that the graph cannot be realized by a transitive subgroup $G\subset S_5^+$.

We have 3 cases here to be studied, as follows:

(1) The graph $A_\infty$ corresponds to $S_5^+$ itself. As for the perturbation of $A_\infty^{(1)}$, this dissapears, because our notion of transitivity requires the subfactor extremality.

(2) For the McKay graphs of $\mathbb Z_5,D_5,GA_1(5),A_5,S_5$ there is nothing to be done, all these graphs being solutions to our problem.

(3)  The possible twists of $A_5,S_5$, coming from the graphs in Proposition 7.8 (3) above. cannot contain $S_5$, because their cardinalities are smaller or equal than $|S_5|=120$.
\end{proof}

In connection now with our maximality questions, we have:

\begin{theorem}
The inclusion $S_5\subset S_5^+$ is maximal.
\end{theorem}

\begin{proof}
This follows indeed from Proposition 7.9 above, with the remark that $S_5$ being transitive, so must be any intermediate subgroup $S_5\subset G\subset S_5^+$.
\end{proof}

With a little more work, the above considerations can give the full list of  transitive subgroups $G\subset S_5^+$. To be more precise, we have here the various subgroups appearing in Proposition 7.9, plus some possible twists of $A_5,S_5$, which remain to be investigated.

\section{Partial maximality}

We have seen in the previous section that, thanks to the results in \cite{bb1}, \cite{max}, \cite{imp}, some of the inclusions coming from Proposition 7.2 are in fact plainly maximal. There are of course many open questions here, and further conjectures that can be made.

In this section we go back to these inclusions, but with a different idea in mind, namely that of investigating them at a higher easiness level. Let us begin with:

\begin{definition}
We say that an inclusion of easy quantum groups $G\subset H$ is maximal at order $p\in\mathbb N$ when there is no proper intermediate subgroup 
$$G\subset K\subset H$$
which is easy at order $p$, in the sense that $K=K^p$.
\end{definition}

Here the quantum groups $K^p$ are those constructed in section 2 above, and appearing in Theorem 2.7. According to the results there, the maximality of $G\subset H$ in the easy setting means to have maximality in the above sense, at order $p=1$.

We also say that an inclusion of easy quantum groups $G\subset K\subset H$ is complete at order $p$ when $K$ is the only intermediate quantum group which is easy at order $p$.

In the case $p=2$, we have the following technical statement:

\begin{proposition}
Assuming that $G\subset H$ comes from an inclusion of categories of partitions $D\subset E$, the maximality at order $2$ is equivalent to the condition
$$<span(D),\alpha T_\pi+\beta T_\sigma>=span(E)$$
for any $\pi,\sigma\in E$, not both in $D$, and for any $\alpha,\beta\neq0$.
\end{proposition}

\begin{proof}
Consider indeed an intermediate category $span(D)\subset C\subset span(E)$, corresponding to an intermediate quantum group $G\subset K\subset H$ having order 2. According to the theory in section 2 above, the order 2 condition means that we have $C=<C\cap span_2(P)>$, where $span_2$ denotes the space of linear combinations having 2 components.

Since we have $span(E)\cap span_2(P)=span_2(E)$, the order 2 formula reads:
$$C=<C\cap span_2(E)>$$

Now observe that the category on the right is generated by the categories $C_{\pi\sigma}^{\alpha\beta}$ constructed in the statement. Thus, the order 2 condition reads:
$$C=\left<C_{\pi\sigma}^{\alpha\beta}\Big|\pi,\sigma\in E,\alpha,\beta\in\mathbb C\right>$$

Now since the maximality at order 2 of the inclusion $G\subset H$ means that we have $C\in\{span(D),span(E)\}$, for any such $C$, we are led to the following condition:
$$C_{\pi\sigma}^{\alpha\beta}\in\{span(D),span(E)\}\quad,\quad\forall\pi,\sigma\in E,\alpha,\beta\in\mathbb C$$

But this gives precisely the condition in the statement.
\end{proof}

Let us study now the inclusions in Proposition 7.2. We first have:

\begin{proposition}
$S_N\subset S_N^+$ is maximal at order $2$.
\end{proposition}

\begin{proof}
We use the ``semicircle capping'' method from \cite{bcs}, where the inclusion $S_N\subset S_N^+$ was shown to be maximal, at the easy quantum group level. To be more precise, it was shown there that any $\pi\in P-NC$ has the property $<\pi>=P$, and in order to establish this formula, the idea was to cap $\pi$ with semicircles, as to preserve one crossing, chosen in advance, and to end up, by a recurrence procedure, with the standard crossing.

In our present case now, if we want to prove the maximality at level 2, in view of Proposition 8.2 above, the statement that we have to prove is as follows: ``for $\pi\in P-NC,\sigma\in P$ and $\alpha,\beta\neq0$ we have $<\alpha T_\pi+\beta T_\sigma>=span(P)$''. 

In order to do this, our claim is that the same method as in \cite{bcs} applies, after some suitable modifications. We have indeed two cases, as follows:

(1) Assuming that $\pi,\sigma$ have at least one different crossing, we can cap as in \cite{bcs} the partition $\pi$ as to end up with the basic crossing, and $\sigma$ becomes in this way an element of $P(2,2)$ different from this basic crossing, and so a noncrossing partition, from $NC(2,2)$. Now by substracting this noncrossing partition, which belongs to $C_{S_N^+}=span(NC)$, we obtain that the standard crossing belongs to $<\alpha T_\pi+\beta T_\sigma>$, and we are done.

(2) In the case where $\pi,\sigma$ have exactly the same crossings, we can start our descent procedure by selecting one common crossing, and then two strings of $\pi,\sigma$ which are different, and then joining the crossing to these two strings. We obtain in this way a certain linear combination $\alpha' T_{\pi'}+\beta'T_{\sigma'}\in <\alpha T_\pi+\beta T_\sigma>$ which satisfies the conditions in (1) above, and we can continuate as indicated there.
\end{proof}

More generally, we have the following result:

\begin{theorem}
The following inclusions are complete at order $2$:
\begin{enumerate}
\item $S_N\subset S_N^+$.

\item $O_N\subset O_N^*\subset O_N^+$.

\item $B_N\subset B_N^+$.
\end{enumerate}
\end{theorem}

\begin{proof}
Here (1) comes from Proposition 8.3, and the proofs of (2) and (3) are similar, by adapting the semicircle capping proofs in \cite{bcs}, in the same way.
\end{proof}

At order 3 and higher the situation is more complicated, because when given a linear combination of type $\alpha T_\pi+\beta T_\sigma+\gamma T_\nu$ with $\pi,\sigma,\nu$ assumed to be distinct, it is quite unclear on what to deduce on the joint crossings of $\pi,\sigma,\nu$, in order to perform the descent method, by capping. All this requires some new ideas, and we believe that the proof of the maximality of the above inclusions, at order 3 or higher, is a good question.

In addition, we know from section 5 above that the inclusion $B_N\subset B_N^+$ is definitely not maximal at order 8, because of the intermediate quantum group $B_N^\circ$ constructed there. Our conjecture would be that this inclusion should be maximal up to order 7.

Summarizing, the notion of partial easiness introduced in this paper leads to some advances on these maximality questions. The main questions raised by the present work, however, remain those raised in sections 3-6, in connection with the exact computation of the easiness level, for the various homogeneous quantum groups $S_N\subset G\subset U_N^+$.

\end{document}